\documentclass[12pt]{amsart}
\usepackage{amsfonts}
\usepackage{amsmath}
\usepackage{amssymb}
\usepackage[margin=1in]{geometry}
\usepackage{tikz}
\usepackage{amsthm}    

\usepackage{mathtools}							
\usepackage{commath}							
\usepackage{enumitem}  																
\usepackage{xurl}	

\usepackage[english]{babel}					
\usepackage[hidelinks]{hyperref}  					

\makeatletter
\@namedef{subjclassname@2020}{\textup{2020} Mathematics Subject Classification}
\makeatother

\newtheorem{theorem}{Theorem}[section]
\newtheorem{corollary}[theorem]{Corollary}
\newtheorem{lemma}[theorem]{Lemma}
\newtheorem{proposition}[theorem]{Proposition}

\theoremstyle{definition}

\newtheorem{remark}[theorem]{Remark}

\newtheorem{question}[theorem]{Question}

\numberwithin{equation}{section}

\newcommand{\R}{\mathbb{R}}

\newcommand{\MP}{\mathcal{P}}

\renewcommand{\Re}{\operatorname{Re}}

\newcommand{\I}{\mathrm{i}}
\newcommand{\e}{\mathrm{e}}

\newcommand{\jm}{j_{\text{max}}}

\newcommand{\eps}{\varepsilon}

\DeclareMathOperator{\Li}{Li}
\DeclareMathOperator{\li}{li}
\DeclareMathOperator{\Var}{Var}

\begin{document}

\title[A new generalized prime random approximation procedure]{A new generalized prime random approximation procedure and some of its applications}

\author[F. Broucke]{Frederik Broucke}
\thanks{F. Broucke was supported by the Ghent University BOF-grant 01J04017}

\author[J. Vindas]{Jasson Vindas}
\thanks {J. Vindas was partly supported by Ghent University through the BOF-grant 01J04017 and by the Research Foundation--Flanders through the FWO-grants 1510119N and G067621N}

\address{Department of Mathematics: Analysis, Logic and Discrete Mathematics\\ Ghent University\\ Krijgslaan 281\\ 9000 Gent\\ Belgium}
\email{fabrouck.broucke@ugent.be}
\email{jasson.vindas@ugent.be}


\begin{abstract}
We present a new random approximation method that yields the existence of a discrete Beurling prime system $\mathcal{P}=\{p_{1}, p_{2}, \dotso\}$ which is very close in a certain precise sense to a given non-decreasing, right-continuous, nonnegative, and unbounded function $F$. This discretization procedure improves an earlier discrete random approximation method due to H.\ Diamond, H.\ Montgomery, and U.\ Vorhauer [Math. Ann. 334 (2006), 1--36], and refined by W.-B.\ Zhang [Math. Ann. 337 (2007), 671--704]. 

We obtain several applications. Our new method is applied to a question posed by M.\ Balazard concerning Dirichlet series with a unique zero in their half plane of convergence, to construct examples of very well-behaved generalized number systems that solve a recent open question raised by  T.\ Hilberdink and A.\ Neamah in [Int.\ J.\ Number Theory 16 05 (2020), 1005--1011], and to improve the main result from [Adv.\ Math.\ 370 (2020), Article 107240], where a Beurling prime system with regular primes but extremely irregular integers was constructed.
\end{abstract}

\maketitle

\section{Introduction}

In their seminal work \cite{DiamondMontgomeryVorhauer}, Diamond, Montgomery, and Vorhauer established the optimality of Landau's abstract prime number theorem\footnote{Landau's original statement is the well-known Prime Ideal Theorem, but his reasoning essentially leads to the first ever known abstract PNT \cite{BatemanDiamond,Knopfmacher}.  } \cite{Landau1903}, partly solving so a long-standing conjecture of Bateman and Diamond \cite[Conjecture 13B, p.~199]{BatemanDiamond}. One of the cornerstones in their arguments is a probabilistic construction, which they developed in order to produce discrete approximations to `continuous prime distribution functions' by random generalized primes.  Refinements were obtained by Zhang in \cite{Zhang2007} (cf. \cite{DiamondZhangbook}). Their discrete random approximation result, from now on referred to as the DMVZ-method, may be summarized as follows.

\begin{theorem}[Diamond, Montgomery, Vorhauer \cite{DiamondMontgomeryVorhauer}, Zhang \cite{Zhang2007}]
\label{th: DMVZ-method}
Let $f$ be a non-negative $L^{1}_{loc}$-function supported on $[1,\infty)$ satisfying 
\begin{equation}
\label{assumptions DMVZ-method}	f(u) \ll \frac{1}{\log u} \quad \text{and} \quad \int_{1}^{\infty}f(u)\dif u = \infty.
\end{equation}
Then there exists an unbounded sequence of real numbers $1<p_1<p_2<\dots<p_j<\dots$ such that for any real $t$ and any $x\ge1$
\begin{equation}
\label{eq: approximation exponential integrals}	\abs[3]{\sum_{p_{j}\le x}p_{j}^{-\I t} - \int_{1}^{x}u^{-\I t}f(u)\dif u} \ll \sqrt{x} + \sqrt{\frac{x\log(\,\abs{t}+1)}{\log(x+1)}}. 
\end{equation}
\end{theorem}
The sequence arising from the DMVZ-method might be regarded as a Beurling prime system. Indeed, following Beurling \cite{Beurling1937} (cf. \cite{BatemanDiamond,DiamondZhangbook}), a set of generalized prime numbers is simply an unbounded non-decreasing sequence of real numbers $\MP=\{p_j\}_{j=1}^{\infty}$ subject to the only requirement $p_1>1$. We denote as $\pi_{\MP}(x)$ the function that counts the number of generalized primes not exceeding a given number $x$. The function $f$ can then be interpreted as a template `prime density measure' $\dif F(u)=f(u)\dif u$, whose continuous `prime distribution function' $F(x)=\int_{1}^{x}f(u)\dif u$ is unbounded and satisfies the Chebyshev upper bound $\ll x/\log x$. The importance of the bound \eqref{eq: approximation exponential integrals} lies in the fact that it is often strong enough for transferring many properties from $\exp\left(\int_{1}^{\infty} x^{-s}\dif F(x) \right)$ into desired analytic properties of the Beurling zeta function associated to $\MP$, that is, 
\[
\zeta_{\MP}(s)= \prod^{\infty}_{j=1} (1-p_{j}^{-s})^{-1}.
\]
We refer to the monograph \cite{DiamondZhangbook} and the recent article \cite{BrouckeDebruyneVindas2020} for relevant applications of the DMVZ-method.

The main goal of this paper is to establish a direct improvement to the DMVZ-method by obtaining a significantly stronger bound for the difference $\pi_{\MP}-F$ than the one delivered by Theorem \ref{th: DMVZ-method}. In fact, setting $t=0$ in \eqref{eq: approximation exponential integrals} yields $\pi_{\MP}(x) - F(x)\ll \sqrt{x}$. We will show that it is possible to select the sequence $\MP$ in such way that the much better bound $\pi_{\MP}(x) - F(x)\ll 1$ holds, as stated in the ensuing theorem, our main result. In addition, our discretization procedure can be applied to approximate measures $\dif F$ that are not necessarily absolutely continuous with respect to the Lebesgue measure.

\begin{theorem}
\label{th: discretization}
Let $F$ be a non-decreasing right-continuous function tending to $\infty$ with $F(1) = 0$ and satisfying the Chebyshev upper bound $F(x) \ll x/\log x$. Then there exists a set of generalized primes $\MP=\{p_j\}_{j=0}^{\infty}$  such that $\abs{\pi_{\MP}(x) - F(x)} \le 2$ and such that for any real $t$ and any $x\ge 1$
\begin{equation}
\label{eq: bound exponential sum}
\abs[3]{ \sum_{p_{j}\le x}p_{j}^{-\I t} - \int_{1}^{x}u^{-\I t}\dif F(u) } \ll \sqrt{x} + \sqrt{\frac{x \log(\,\abs{t}+1)}{\log(x+1)}}. 
\end{equation}

If in addition $F$ is continuous, the sequence $\MP$ can be chosen to be  (strictly) increasing and such that $\abs{\pi_{\MP}(x) - F(x)} \le 1$.
\end{theorem}

The proof of Theorem \ref{th: discretization} will be given in Section \ref{sec: the main result}. The essential difference between the DMVZ probabilistic scheme and our proof is that we make a completely different choice of how the generalized prime random variables are distributed in order to generate the discrete random approximations, allowing for a more accurate control on the size of the difference $\pi_{\MP}(x)-F(x)$. 

The rest of the article is devoted to illustrating the usefulness of Theorem \ref{th: discretization} through three applications. In all these applications, the stronger bound $\pi_{\MP}(x)-F(x) \ll1$ instead of $\pi_{\MP}(x)-F(x)\ll \sqrt{x}$ plays a crucial role. Our first application concerns a question posed by M.\ Balazard (we consider a strengthened version of \cite[Open Problem 24]{MontgomeryVorhauer2000}):
\begin{question}
\label{Balazard's question}
Does there exist a Dirichlet series $\sum_{n=1}^{\infty} a_n n^{-s}$ which has exactly one zero in its half plane of convergence?
\end{question}
This question is motivated by the fact that if the Riemann hypothesis is true, the Dirichlet series
\begin{equation}
\label{eq: classical Moebius}
\sum_{n=1}^{\infty} \frac{\mu(n)}{n^{s}}=\frac{1}{\zeta(s)},
\end{equation}
 where $\mu$ is the (classical) M\"obius function, provides an example of such a Dirichlet series, since it would have a unique zero, namely at $s=1$, in its half plane of convergence $\{s:\: \Re s >1/2\}$. The idea is of course to find an unconditional example. We are not able to answer Question \ref{Balazard's question} here for Dirichlet series as in its statement, but, armed with Theorem \ref{th: discretization}, we will prove that Balazard's question can be affirmatively answered for \emph{general} Dirichlet series.
 
 \begin{proposition}
 \label{prop: Balazard's question}
  There are an unbounded sequence $1=n_0<n_1\leq n_{2}\leq \dots \leq n_{k}\leq \dots $ and a general Dirichlet series of the form 
 \[D(s)=\sum_{k=0}^{\infty} \frac{a_{k}}{n_{k}^{s}}, \qquad \mbox{with } a_k\in \{-1,0,1\},
\]  such that $D(s)$ has abscissa of convergence $\sigma_{c}=1/2$ and has a unique zero on $\{s:\: \Re s >1/2\}$, which is located at $s=1$.
 \end{proposition}

Our example of a general Dirichlet series satisfying the requirements of Proposition \ref{prop: Balazard's question} arises from a Beurling prime system that we shall construct in Section \ref{sec: First application}. This example is actually the Beurling analog of the Dirichlet series \eqref{eq: classical Moebius}. It turns out that the same constructed generalized primes yield a second application, as this generalized number system also provides a positive answer to a recent open problem raised by Hilberdink and Neamah (cf. \cite[Section 4. Open Problem (1)]{NeamahHilberdink}) on the existence of well-behaved Beurling number systems of a certain best possible type; see Section \ref{sec: First application} for details.

As a third application, we conclude this article with an improvement to a result recently established by the authors together with G.\ Debruyne \cite[Theorem 1.1]{BrouckeDebruyneVindas2020}. In that paper, the authors showed the existence of a prime system with regular primes (i.e., satisfying the RH in the form of the PNT with remainder $O(\sqrt{x})$) but with integers displaying extreme oscillation. Using Theorem \ref{th: discretization}, we will be able to improve in Section \ref{sec: third application} the regularity of the primes from an $O(\sqrt{x})$ error term to one of order $O(1)$, while still keeping the same irregularity of the integers.

\section{The main result}
\label{sec: the main result}

This section is devoted to a proof of Theorem \ref{th: discretization}. Like in the DMVZ-method, our starting point is a probabilistic inequality for sums of random variables essentially due to Kolmogorov (see e.g.\ \cite[Chapter V]{Loeve}), which we shall employ to bound the probability of certain events. The following inequality is a slight variant of \cite[Lemma 8]{DiamondMontgomeryVorhauer}, and the proof given there can readily  be adapted to yield the ensuing form of the lemma.
\begin{lemma}
\label{lem: probabilistic inequality}
For $1\le j\le J$, let $X_{j}$ be independant random variables with $E(X_{j})=0$, $\abs{X_{j}} \le 2$, and $\Var(X_{j}) = \sigma_{j}^{2}$. Let $S=\sum_{j=1}^{J}X_{j}$, and $\sigma^{2} = \sum_{j=1}^{J}\sigma_{j}^{2} = \Var(S)$. Then
\[	
	P(S\ge v) \le \begin{dcases}
		\exp\biggl(-\frac{v^{2}}{4\sigma^{2}}\biggr)		&\text{if } v\le u_{0}\sigma^{2};\\
		\exp\biggl(-\frac{u_{0}v}{4}\biggr)			&\text{if } v>u_{0}\sigma^{2}.
	\end{dcases}
\]
Here $u_{0}$ is the positive solution of $\e^{u} = 1+u+u^{2}$, $u_{0} \approx 1.79328$.
\end{lemma}

\begin{proof}[Proof of Theorem \ref{th: discretization}]
Write $\dif F = \dif F_{c} + \dif F_{d}$, where $\dif F_{c}$ is a continuous measure, and $\dif F_{d}$ is purely discrete:
\[ \dif F_{d} = \sum_{n=1}^{\infty}\alpha_{n}\delta_{y_{n}}, \qquad y_{n}>1, \quad \alpha_{n} > 0, \]
where $\delta_{y}$ denotes the Dirac measure concentrated at $y$ and the sequence $\{y_n\}_{n=1}^{\infty}$ consists of distinct points.  We will discretize both measures separately\footnote{$\dif F_{d}$ is already a purely discrete measure, but does not necessarily arise as the prime counting measure of a discrete Beurling prime system, since $\{y_n\}_{n=1}^{\infty}$ may have accumulation points, and since, even if this sequence happens to be discrete, we do not assume that the $\alpha_{n}$ are integers.}. Let us start with the continuous part. 

Set $q_{0}=1$, $q_{j} = \min\{ x : F_{c}(x) = j\}$, for $j < \jm$, where $\jm=\infty$ if $F_{c}(\infty) = \infty$, and $\jm = \lceil F_{c}(\infty) \rceil$ if $F_{c}(\infty) < \infty$. Let $\{P_{j}\}_{1\le j<\jm}$ be a sequence of independent random variables, where $P_{j}$ is distributed on $(q_{j-1},q_{j}]$ according to the probability measure $\dif F_{c}|_{(q_{j-1},q_{j}]}$. Fix a number $t \in \R$ and set $X_{j,t} = \cos(t\log P_{j})$. For such a fixed $t$, the $X_{j,t}$ are independent random variables with expectation 
\[	E(X_{j,t}) = \int_{q_{j-1}}^{q_{j}}\cos(t\log u)\dif F_{c}(u) \]
and variance $\Var(X_{j,t}) \le 1$. Let $C$ be a constant such that 
\[	 F_{c}(x) \le C\frac{x}{\log (x+1)}. \]
Let $J<\jm$ and suppose that $q_{J}/\log(q_{J}+1) \ge \log(\,\abs{t}+1)$. Set $x=q_{J}$ and let $D=\max\{\sqrt{8C}, 8/u_{0}\}$, where $u_{0}$ is the number appearing in Lemma \ref{lem: probabilistic inequality}. Applying that lemma to the random variables $X_{j,t}-E(X_{j,t})$, with $v=D(\sqrt{x} + \sqrt{x\log(\,\abs{t}+1)/\log(x+1)})$, we get
\begin{multline*}	
	P\Biggl(\sum_{j=1}^{J}\cos(t\log P_{j}) - \int_{1}^{x}\cos(t\log u)\dif F_{c}(u) \ge D\biggl(\sqrt{x} + \sqrt{\frac{x\log(\,\abs{t}+1)}{\log(x+1)}}\biggr)\Biggr) \le \\
	\max\left\{\exp\biggl(-\frac{D^{2}}{4\sigma^{2}}\biggl(x+\frac{x\log(\,\abs{t}+1)}{\log(x+1)}\biggr)\biggr); \exp\biggl(-\frac{u_{0}}{4}D\biggl(\sqrt{x} + \sqrt{\frac{x\log(\,\abs{t}+1)}{\log(x+1)}}\biggr)\biggr)\right\}.
\end{multline*}
Here 
\[
	\sigma^{2} = \sum_{j=1}^{J}\Var(X_{j,t}) \le J =  F_{c}(x) \le C\frac{x}{\log (x+1)} \quad \text{and} \quad \sqrt{\frac{x}{\log(x+1)}} \ge \sqrt{\log(\,\abs{t}+1)}.
\]
Hence the above probability is bounded by $(x+1)^{-2}(\,\abs{t}+1)^{-2}$. Applying the same argument to the random variables $-X_{j,t}$, $\pm Y_{j,t} = \pm\sin(t\log P_{j})$, we get the same bounds for the corresponding probabilities.

Let 
\[	S(x,t) = \sum_{P_{j}\le x}P_{j}^{-\I t}, \quad S_{c}(x,t) = \int_{1}^{x}u^{-\I t}\dif F_{c}(u). \]
Then for $x=q_{J}$ with $x/\log(x+1) \ge \log(\,\abs{t}+1)$
\[
	P\Biggl(\, \abs{S(x,t) - S_{c}(x,t)} \ge \sqrt{2}D\biggl(\sqrt{x}+\sqrt{\frac{x\log(\, \abs{t}+1)}{\log(x+1)}}\biggr)\Biggr) \le \frac{4}{(x+1)^{2}(\,\abs{t}+1)^{2}}.
\]
Let 
$j_{t} = \min\{ j<\jm: q_{j}/\log(q_{j}+1) \ge \log(\,\abs{t}+1)\}$, where we set $j_{t}=\infty$ when the set is empty (which may happen if $\jm<\infty$ and $t$ is sufficiently large). Let $A_{k,j}$ denote the event $\abs{S(q_{j},k)-S_{c}(q_{j},k)} \ge \sqrt{2}D(\sqrt{q_{j}} + \sqrt{q_{j}\log(k+1)/\log(q_{j}+1)})$. Since\footnote{If $j_{k}=\infty$, the inner sum is zero by convention.}
\[
	\sum_{k=1}^{\infty}\sum_{j_{k}\le j < \jm}P(A_{kj}) \le \sum_{k=1}^{\infty}\sum_{j_{k}\le j< \jm}\frac{4}{(q_{j}+1)^{2}(\,\abs{k}+1)^{2}}\ll \sum_{k=1}^{\infty}\sum_{j=1}^{\infty} \frac{1}{j^{2}k^2} <\infty,
\]
the Borel-Cantelli lemma implies that the probability that infinitely many of the events $A_{k,j}$, $k\ge 1$, $j_{k}\le j < \jm$, occur, is zero. Fix now a point $\omega$ of the probability space which is only contained in finitely many $A_{k,j}$ with $k\ge 1$ and $j_{k}\le j < \jm$ and set $p_{j} = P_{j}(\omega)$. Then there exists a $k_{0}\ge 1$ such that for every $k\ge k_{0}$ and any $j$ with $j_{k} \le j<\jm$ (with now $S(x,k) = S(x,k)(\omega)$)
\begin{equation}
\label{eq: bound exponential sum for integers}
	S(q_{j},k) - S_{c}(q_{j},k) \ll \sqrt{q_{j}} + \sqrt{\frac{q_{j}\log(k+1)}{\log(q_{j}+1)}}.
\end{equation}
Also by construction of the random variables $P_{j}$, we have 
\[
	\abs{\pi_{\MP}(x) -  F_{c}(x)} \le 1, \quad \text{where } \pi_{\MP}(x) = \sum_{p_{j}\le x}1.
\]
Let now $1\le k < k_{0}$ and $j<\jm$ arbitrary. Integrating by parts,
\[
	\abs{S(q_{j},k) - S_{c}(q_{j},k)} = \abs{\int_{1}^{q_{j}^{+}}u^{-\I k}\dif\,(\pi_{\MP}(u) -  F_{c}(u))} \ll 1 + k_{0}\int_{1}^{q_{j}}\frac{\dif u}{u} \ll \log q_{j}.
\] 
If $k\ge k_{0}$ and $j<\min\{j_{k}, \jm\}$, then  
\[
	\abs{S(q_{j},k) - S_{c}(q_{j},k)} = \abs{\sum_{l=1}^{j}\int_{q_{l-1}}^{q_{l}}(p_{l}^{-\I k} - u^{-\I k})\dif  F_{c}(u)} \le 2 F_{c}(q_{j})\ll \sqrt{\frac{q_{j}}{\log(q_{j}+1)}}\sqrt{\log(k+1)}.
\]
We conclude that the bound \eqref{eq: bound exponential sum for integers} holds for any $k\ge0$ and any $j<\jm$.

Suppose now that $k\ge1$ and that for some $1\le j<\jm$, $x \in (q_{j-1},q_{j}]$. Then 
\begin{align*}
	S(x,k) = S(q_{j-1},k) + O(1)	&= S_{c}(q_{j-1},k) + O\biggl(\sqrt{q_{j-1}}+\sqrt{\frac{q_{j-1}\log (k+1)}{\log(q_{j-1}+1)}}\biggr) +O(1) \\
							&= S_{c}(x,k) + O\biggl(\sqrt{x}+\sqrt{\frac{x\log (k+1)}{\log(x+1)}}\biggr).
\end{align*}	 
If $\jm<\infty$ and $x>q_{\jm}$, then 
\begin{align*}
	S(x,k) 	&= S(q_{\jm},k) = S_{c}(q_{\jm},k) + O\biggl(\sqrt{q_{\jm}}+ \sqrt{\frac{q_{\jm}\log (k+1)}{\log(q_{\jm}+1)}}\biggr) \\
			&= S_{c}(x,k) + O\biggl(\sqrt{x} + \sqrt{\frac{x\log(k+1)}{\log(x+1)}}\biggr). 
\end{align*}
If $t\in[k,k+1]$ for some $k\ge0$, then by integration by parts,
\begin{align*}
S(x,t) 	&= \int_{1}^{x^{+}}u^{-\I(t-k)}\dif S(u,k) = S(x,k)x^{-\I(t-k)} + \I(t-k)\int_{1}^{x}S(u,k)u^{-\I(t-k)-1}\dif u \\
		&= S_{c}(x,k)x^{-\I(t-k)} + \I(t-k)\int_{1}^{x}S_{c}(u,k)u^{-\I(t-k)-1}\dif u + O\biggl(\sqrt{x}+ \sqrt{\frac{x\log(t+1)}{\log(x+1)}}\biggr) \\
		&= S_{c}(x,t) + O\biggl(\sqrt{x}+ \sqrt{\frac{x\log(t+1)}{\log(x+1)}}\biggr).
\end{align*}
Finally for negative $t$ we obtain the same bounds by taking the complex conjugate.

In order to discretize $\dif  F_{d}$, we can apply the same idea, but with a slight modification, since it may not be possible to partition $[1,\infty)$ into disjoint intervals each having total mass $1$. 
We proceed as follows. Set $q_{0}=1$, $q_{j} = \min\{x:  F_{d}(x) \ge j\}$, for $1\le j< \jm$, where again $\jm = \infty$ if $F_{d}(\infty)=\infty$ and $\jm=\lceil F_{d}(\infty)\rceil$ if $F_{d}(\infty) < \infty$. Note that it may occur that $q_{j} = q_{j+1} = \dotso = q_{j+k}$ for some $k\ge 1$; we have $q_{j}<q_{j+1} \iff \lfloor F_{d}(q_{j})\rfloor = j$. We will distribute the masses $\alpha_{n}$ over the intervals $[q_{j-1},q_{j}]$, $0\le j < \jm$ in such a way that each interval $[q_{j-1},q_{j}]$ has mass $1$. At points $q_{j}$, where $F_{d}$ ``spills over'' the next integer (or next $k+1$ integers), we divide the mass $\alpha$ of the point $q_{j}$ as $\alpha = \beta + k + \gamma$, where $\beta$ is ``given'' to the interval $[q_{j-1},q_{j}]$, and $\gamma$ is ``given'' to $[q_{j+k}, q_{j+k+1}]$. Making this precise, set $\gamma_{0} = 0$ and if $\gamma_{j-1}$ is defined with $j < \jm$, define numbers $\beta_{j}, \gamma_{j+k}\in [0,1]$ as  
\[
	\beta_{j} = 1 - \gamma_{j-1} - \sum_{q_{j-1}<y_n<q_{j}}\alpha_{n}, \quad \gamma_{j+k} = F_{d}(q_{j+k}) - \lfloor F_{d}(q_{j+k})\rfloor = F_{d}(q_{j+k}) - (j+k),
\]
where $k$ is the largest number (possibly zero) such that $q_{j} = q_{j+1}  = \dotso = q_{j+k}$. Note that the sum over $\alpha$ can be empty (hence zero), but may also consist of infinitely many terms.

Let $\{P_{j}\}_{1\le j<\jm}$ be a sequence of independent discrete random variables, where $P_{j}$ is distributed according to the probability law
\[
	P(P_{j} = y_{n}) = \begin{dcases}
					\gamma_{j-1}	&\text{if } y_n=q_{j-1}, \\
					\alpha_{n}		&\text{if } q_{j-1} < y_{n}< q_{j},\\
					\beta_{j}		&\text{if } y_n=q_{j};					
	\end{dcases}
\]
in the case that $q_{j-1}<q_{j}$, and $P_{j}$ is distributed according to the trivial law $P(P_{j} = q_{j})=1$ in the case that $q_{j-1} = q_{j}$. Note that when $q_{j-1}<q_{j}$ it can happen that $q_{j-1}$ or $q_{j}$ do not occur in the sequence $\{y_{n}\}_{n=1}^{\infty}$; however in these cases one sees that $\gamma_{j-1}=0$ and $\beta_{j}=0$ respectively. Again we consider for fixed $t$ the independent random variables $X_{j,t} = \cos(t\log P_{j})$. Let $J < \jm$ be such that $q_{J} < q_{J+1}$ or $J=\jm-1$. Suppose also that $q_{J}/\log(q_{J}+1) \ge \log(\,\abs{t}+1)$, and set $x = q_{J}$. We apply Lemma \ref{lem: probabilistic inequality} to the random variables $X_{j,t} - E(X_{j,t})$; however, in this case 
\[	
	\sum_{j=1}^{J}E(X_{j,t}) = \sum_{y_{n}\le x}\alpha_{n}\cos(t\log y_{n}) - \gamma_{J}\cos(t\log q_{J}).
\]
Nevertheless, we can absorb the last term in the error term by multiplying it by 2:
\[
	P\Biggl(\sum_{j=1}^{J}X_{j,t} - \sum_{y_{n}\le x}\alpha_{n}\cos(t\log y_{n}) \ge 2v\Biggr) \le P\Biggl(\sum_{j=1}^{J}\bigl(X_{j,t} - E(X_{j,t})\bigr) \ge v\Biggr)
\]
for $v\ge1$. Applying Lemma \ref{lem: probabilistic inequality} with $v = D'(\sqrt{x} + \sqrt{x\log(\,\abs{t}+1)/\log(x+1)})$, where $D'=\max\{\sqrt{8C'}, 8/u_{0}\}$, with $C'$ a constant such that $ F_{d}(x) \le C'x/\log(x+1)$, we obtain
\[
	P\Biggl(\sum_{j=1}^{J}X_{j,t} - \sum_{y_{n}\le x}\alpha_{n}\cos(t\log y_{n}) \ge 2D'\biggl(\sqrt{x} + \sqrt{\frac{x\log(\,\abs{t}+1)}{\log(x+1)}}\biggr)\Biggr) \le \frac{1}{(x+1)^{2}(\,\abs{t}+1)^{2}}.
\]
The proof can now be completed, mutatis mutandis, as in the continuous case.
\end{proof}

We now show that under the assumption that $F$ is absolutely continuous on any finite interval, we can ensure that the approximating discrete primes are supported on strictly increasing sequences which tend to $\infty$ sufficiently slowly, while still having the bound $\pi_{\MP}(x) - F(x) \ll 1$ instead of the weaker $\pi_{\MP}(x)-F(x)\ll \sqrt{x}$ delivered by the DMVZ-method. The following corollary is a direct improvement to \cite[Lemma 4]{Zhang2007}.

\begin{corollary}
Suppose $f$ is a non-negative $L^{1}_{loc}$-function supported on $[1,\infty)$ and satisfying the conditions \eqref{assumptions DMVZ-method}. Let 
\[	1 < v_{1} < \dotso <v_{k}< v_{k+1}<\dotso, \quad v_{k}\to \infty, \]
be a sequence such that $v_{k+1}-v_{k}\ll \log v_{k}$ and such that for any $t\ge0$
\[
	\sum_{v_{k}\ge h(t)}\frac{(v_{k}-v_{k-1})^{2}}{v_{k}\log v_{k}} \ll \frac{\log(t+1)}{t}, \quad \text{where} \quad h(t) = \log(t+1)\log\log(t+\e). 
\]
Then there exists a generalized prime system $\mathcal{P}=\{p_{j}\}_{j=1}^{\infty}$ supported\footnote{Strictly speaking, $\{p_{j}\}_{j=1}^{\infty}$ needs not be a subsequence of $\{v_{k}\}_{k=1}^{\infty}$, since some primes $p_{j}$ may be repeated.} on the sequence $\{v_{k}\}_{k=1}^{\infty}$ such that for any $x\ge1$ and any $t$
\begin{equation}
\label{eq: prime system bound}
	\abs[3]{\pi_{\MP}(x) - \int_{1}^{x}f(u)\dif u} \ll 1 \quad \text{and} \quad \abs[3]{\sum_{p_{j}\le x}p_{j}^{-\I t} - \int_{1}^{x}u^{-\I t}f(u)\dif u} \ll \sqrt{x} + \sqrt{\frac{x\log(\,\abs{t}+1)}{\log(x+1)}}.
\end{equation}
\end{corollary}
Some examples of admissible sequences are $v_{k} = (\log(k+k_{0}))^{a}(\log \log (k+k_{0}))^{b}$ with  $0<a<1$ and $b\in\mathbb{R}$ and
$v_{k} = \log(k+k_{0})(\log\log(k+k_{0}))^{b}$ with $b\le1$.
\begin{proof} Write $\dif F(u)=f(u)\dif u$. 
The idea of the proof is to construct an `intermediate' measure $\dif G$ which is close to $\dif F$ and supported on the sequence $\{v_{k}\}_{k=1}^{\infty}$. The primes $p_{j}$ will then be obtained discretizing $\dif G$  by using Theorem \ref{th: discretization}.

We set $v_{0} = 1$ and define the measure $\dif G$ as
\[
	\dif G = \sum_{k=1}^{\infty}\alpha_{k}\delta_{v_{k}}, \quad \text{where} \quad \alpha_{k} = \int_{v_{k-1}}^{v_{k}}\dif F.
\] 
By the first requirement on the sequence $\{v_{k}\}_{k=1}^{\infty}$ and the bound $\dif F(u) \ll \dif u/\log u$,  we have $ G(x) - F(x) \ll 1$. Let now $t$ be arbitrary, and let $x$ be such that $x/\log(x+1) < \log(\,\abs{t}+1)$. Then trivially
\[
	\abs[3]{\sum_{v_{k}\le x}\alpha_{k}v_{k}^{-\I t} - \int_{1}^{x}u^{-\I t}\dif F(u)} \leq 2 F(x) \ll \frac{x}{\log(x+1)} < \sqrt{\frac{x\log(\,\abs{t}+1)}{\log(x+1)}}.
\]
If on the other hand $x/\log(x+1) \ge \log(\,\abs{t}+1)$, we proceed as follows:
\[
	\abs[3]{\sum_{v_{k}\le x}\alpha_{k}v_{k}^{-\I t} - \int_{1}^{x}u^{-\I t}\dif F(u)} \ll 1 + \int_{1}^{v_{L}}\dif F(u) + \sum_{k=L+1}^{K}\int_{v_{k-1}}^{v_{k}}\abs{v_{k}^{-\I t}-u^{-\I t}}\dif F(u).
\]
Here $K$ is such that $v_{K}\le x < v_{K+1}$, and $L$ is the largest integer $\le K$ such that $v_{L} < h(\,\abs{t}) = \log(\,\abs{t}+1)\log\log(\,\abs{t}+\e)$. Bounding $\abs[1]{v_{k}^{-\I t}-u^{-\I t}}$ by $\abs{t}(v_{k}-v_{k-1})/v_{k}$ (note that $v_{k}/v_{k-1} \ll 1$) and using the bound $\dif F(u) \ll \dif u/\log u$, we get
\begin{align*}
	\abs[3]{\sum_{v_{k}\le x}\alpha_{k}v_{k}^{-\I t} - \int_{1}^{x}u^{-\I t}\dif F(u)}	
		&\ll \frac{v_{L}}{\log(v_{L}+1)} + \abs{t}\sum_{v_{k}\ge h(\,\abs{t})}\frac{(v_{k}-v_{k-1})^{2}}{v_{k}\log v_{k}}\\
		&\ll \log(\,\abs{t}+1) \le \sqrt{\frac{x\log(\,\abs{t}+1)}{\log(x+1)}},
\end{align*}	
where we used the second property of the sequence $\{v_{k}\}_{k=1}^{\infty}$ and $\log(\,\abs{t}+1) \le x/\log(x+1)$. Applying Theorem \ref{th: discretization} to $G$ yields a sequence $\{p_{j}\}_{j=1}^{\infty}$ of primes satisfying \eqref{eq: prime system bound} (by comparing with $\dif G$ via the triangle inequality). By construction of the discrete random variables in the proof of Theorem \ref{th: discretization}, the primes $p_{j}$ are contained in the support of $\dif G$, that is the sequence $\{v_{k}\}_{k=1}^{\infty}$.
\end{proof}

\begin{remark}
It is possible to generalize Theorem \ref{th: discretization} to functions $F$ with different growth. Indeed, suppose that $F(x) \ll A(x)$, where $A$ is non-decreasing, has tempered growth, namely, $ A(x)\ll  x^{n}$ for some $n$, and satisfies 
\begin{equation}
\label{eq: integral condition A}
\int_{1}^{x}\frac{\sqrt{A(u)}}{u}\dif u \ll \sqrt{A(x)}
\end{equation}
(which implies $x^{\delta} \ll A(x)$ for some $\delta>0$ depending on the implicit constant above). Then the conclusion of theorem holds if we replace the bound \eqref{eq: bound exponential sum} by 
\[
	\abs[3]{ \sum_{p_{j}\le x}p_{j}^{-\I t} - \int_{1}^{x}u^{-\I t}\dif F(u) } \ll \sqrt{A(x)}\bigl(\sqrt{\log(x+1)} + \sqrt{\log(\,\abs{t}+1)}\bigr). 
\] 
We remark that \eqref{eq: integral condition A} is satisfied whenever $A$ is of positive increase (see \cite[Theorem 2.6.1(b) and Definition of PI in p. 71]{BGTbook}). 
\end{remark}

\section{Balazard's question for general Dirichlet series and the existence of well-behaved Beurling numbers of type $[0,1/2,1/2]$}
\label{sec: First application}
In this section we simultaneously give a proof of Proposition \ref{prop: Balazard's question} and address an open question from \cite{NeamahHilberdink}. In preparation, we need to introduce some concepts.

Let  $\mathcal{P}=\{p_{j}\}_{j=0}^{\infty}$ be a Beurling generalized prime system. 
The associated (multi)set of Beurling generalized integers $\mathcal{N}$ is the semi-group generated by 1 and the numbers $p_{j}$, which we arrange in a non-decreasing fashion (taking multiplicities into account): $1=n_{0} < n_{1}\le n_{2} \le \dotso \le n_{k} \le \dotso$. Besides $\pi_{\MP}$ and $\zeta_{\MP}$, we can associate to the number system familiar number theoretic functions \cite{DiamondZhangbook}. The counting function of the generalized integers is denoted as  $N_{\MP}(x) = \sum_{n_{k} \le x}1.$ As in classical number theory, one defines the Riemann prime counting function as 
\[	\Pi_{\MP}(x) = \sum_{p_{j}^{\nu}\le x} \frac{1}{\nu} = \sum_{\nu=1}^{\infty}\frac{\pi_{\MP}(x^{1/\nu})}{\nu}. \]
The functions $N_{\MP}$ and $\Pi_{\MP}$ are then linked via the zeta function identity
\begin{equation}
\label{eq: zeta identity}
	\zeta_{\mathcal{P}}(s) = \int_{1^{-}}^{\infty}x^{-s}\dif N_{\MP}(x) = \exp\biggl(\int_{1^{-}}^{\infty}x^{-s}\dif\Pi_{\MP}(x)\biggr).
\end{equation}

The M\"{o}bius function of the generalized number system is determined by its sum function $M(x)=\sum_{n_{k}\leq x} \mu(n_{k})$,
where $\dif M$ is defined as the (multiplicative) convolution inverse of $\dif N$; equivalently, 
\begin{equation}
\label{eq: Moebius general Dirichlet series}
\sum_{k=1}^{\infty}\frac{\mu(n_k)}{n_k^{s}}=\frac{1}{\zeta_{\MP}(s)}.
\end{equation}

Let us assume that $\zeta_{\MP}$ has abscissa of convergence 1. Following\footnote{We count the primes using Riemann's counting function $\Pi$ instead of Chebyshev's $\psi$. An error term for $\Pi$ can be transported to one for $\psi$ at just the cost of an additional $\log$-factor.} Hilberdink and Neamah (cf. \cite{NeamahHilberdink}), we define the three numbers $\alpha, \beta, \gamma$ as the unique exponents (necessarily elements of $[0,1]$) for which the relations
\begin{subequations}
\begin{align*}
	\Pi(x) 	&= \Li(x) + O(x^{\alpha+\eps}),\\
	N(x)		&= \rho x + O(x^{\beta+\eps})\\
	M(x)		&= O(x^{\gamma+\eps}),
\end{align*}
\end{subequations}
hold for some $\rho>0$ and for any $\eps>0$, but no $\eps<0$. Here we choose to normalize the logarithmic integral as 
\[
 \Li(x) \coloneqq \int_{1}^{x}\frac{1-u^{-1}}{\log u} \dif u.
\]
We then call such a Beurling generalized number system an $[\alpha,\beta,\gamma]$-system. The main result\footnote{For this result it is imperative to consider \emph{discrete} number systems, since it is obviously false for continuous ones: consider for example $\Pi_{0}=\Li$, for which \cite{DiamondZhangbook} $N_0(x)=x$ and $M_{0}(x) = 1-\log x$, for an easy counterexample.} of \cite{NeamahHilberdink} (see also \cite{Hilberdink2005}) tells us that $\Theta=\max\{\alpha,\beta,\gamma\}$ is at least 1/2 and that at least two of these numbers must be equal to $\Theta$. Hilberdink and Lapidus \cite{HilberdinkLapidus2006} call a Beurling number system \emph{well-behaved}\footnote{To ensure this it suffices to know that just two of the numbers are $<1$, as we can deduce from \cite[Theorem 2.3]{HilberdinkLapidus2006} and (the proof of) \cite[Theorem 2.1]{NeamahHilberdink}.} if $\Theta<1$. 

The best possible types of well-behaved generalized numbers are then of type $[0,1/2,1/2]$, $[1/2,0,1/2]$, and $[1/2,1/2,0]$. If the RH holds, then the rational integers are a $[1/2,0,1/2]$-system, so that we have a candidate example of this instance. It is conjectured in \cite{NeamahHilberdink} that there are no $[1/2,1/2,0]$-systems. The following open question is also posed in \cite[Section 4]{NeamahHilberdink}: Does there exist a $[0,\beta,\beta]$ system with $\beta<1$? The following theorem answers this question positively; we actually establish the existence of $[0,1/2,1/2]$-systems.

\begin{theorem}
\label{th: [0,1/2]-system}
There is a discrete Beurling generalized prime system $\mathcal{P}$ 
such that
\begin{equation}
\label{eq: behavior Pi}
	\Pi_{\MP}(x) = \Li(x) + O(\log\log x),
\end{equation}
\begin{equation}
\label{eq: behavior N and M}
	N_{\MP}(x) = x + O\bigl(x^{1/2}\exp(c(\log x)^{2/3})\bigr), \quad M_{\MP}(x) = O\bigl(x^{1/2}\exp(c(\log x)^{2/3})\bigr),
\end{equation}
for some $c>0$, and
\begin{equation}
\label{eq: Omega-behavior N and M}
N_{\MP}(x) = x + \Omega_{\varepsilon}\bigl(x^{1/2-\varepsilon}\bigr), \quad M_{\MP}(x) = \Omega_{\varepsilon} \bigl(x^{1/2-\varepsilon}\bigr),
\end{equation}
for any $\varepsilon>0$. 

\end{theorem}

It follows at once that \eqref{eq: Moebius general Dirichlet series} for the Beurling number system from Theorem \ref{th: [0,1/2]-system} furnishes an example of a general Dirichlet series having abscissa of convergence $\sigma_{c}=1/2$ and with a unique zero in its half plane of convergence, namely, at $s=1$, which proves Proposition \ref{prop: Balazard's question}.

\begin{proof}
We apply Theorem \ref{th: discretization} to $F(x) = \li(x)$, where $\li$ is such that $\Li(x) = \sum_{\nu\ge1}\li(x^{1/\nu})/\nu$. A small computation shows that
\[	
	\li(x) = \sum_{\nu=1}^{\infty}\frac{\mu(\nu)}{\nu} \Li (x^{1/\nu}) =\sum_{n=1}^{\infty}\frac{(\log x)^{n}}{n! n \zeta(n+1)}. 
\]
Here $\zeta$ and $\mu(\nu)$ are the classical Riemann-zeta and M\"{o}bius functions, respectively. The Chebyshev bound holds since 
\[	
	\li(x) \le  \sum_{n=1}^{\infty}\frac{(\log x)^{n}}{n! n} = \Li(x) \le 2\frac{x}{\log x}, \quad \text{if } x \gg 1. 
\]
We thus find generalized primes $\mathcal{P} :\: 1<p_{1}<p_{2}<\dotso$ with $\pi_{\MP}(x) = \sum_{p_{j}\le x}1 = \li(x) + O(1)$ and satisfying \eqref{eq: bound exponential sum}. To simplify the notation, we drop the subscript $\MP$ from all counting functions associated to this generalized prime system, but we make an exception with $\zeta_{\MP}(s)$ for which the subscript is kept in order to distinguish it from the Riemann zeta function $\zeta(s)$. The Riemann prime counting function $\Pi$ of $\MP$ satisfies
\begin{align*}
\Pi(x) 	&= \sum_{\nu=1}^{\left\lfloor\frac{\log x}{\log p_{1}}\right\rfloor}\frac{1}{\nu}\pi(x^{1/\nu}) = \sum_{\nu=1}^{\left\lfloor\frac{\log x}{\log p_{1}}\right\rfloor}\frac{1}{\nu}\bigl( \li(x^{1/\nu}) + O(1)\bigr) \\
		&= \sum_{\nu=1}^{\infty}\sum_{n=1}^{\infty}\frac{(\log x)^{n}}{n!n\zeta(n+1)}\frac{1}{\nu^{n+1}} - \sum_{\nu > \left\lfloor\frac{\log x}{\log p_{1}}\right\rfloor}\sum_{n=1}^{\infty}\frac{(\log x)^{n}}{n!n\zeta(n+1)}\frac{1}{\nu^{n+1}} + O(\log\log x) \\
		&= \Li(x) + O(\log\log x).
\end{align*}
Also 
\[	
	\Li(x) = \li(x) + O\biggl(\frac{\sqrt{x}}{\log x}\biggr), \quad \text{so } \Pi(x) = \pi(x) + O\biggl(\frac{\sqrt{x}}{\log x}\biggr). 
\]
The bound $\Pi(x) - \Li(x) \ll \log\log x$ implies that $Z(s) \coloneqq \log\zeta_{\mathcal{P}}(s) - \log(s/(s-1))$ has analytic continuation to the half plane $\sigma>0$. By changing a finite number of primes, we may assume that $Z(1)=0$, so that the corresponding integers have density $1$. 
Using the bound \eqref{eq: bound exponential sum} we can deduce good bounds for $Z$ in the half plane $\sigma>1/2$, which allows one to deduce the asymptotic relations \eqref{eq: behavior N and M} via Perron inversion. The proof is essentially the same as that of Zhang's theorem \cite[Theorem 1]{Zhang2007}, but we will repeat it for convenience of the reader.

We have that 
\[
	Z(s) = \int_{1}^{\infty}x^{-s}\dif\, (\pi(x)-\li(x)) + \int_{1}^{\infty}x^{-s}\dif\,(\Pi(x) - \pi(x)) - \int_{1}^{\infty}x^{-s}\dif\,(\Li(x) - \li(x)).
\]
The last two integrals have analytic continuation to $\sigma>1/2$ and are $O((\sigma-1/2)^{-1})$ for $\sigma>1/2$. The first integral has analytic continuation to $\sigma>1/2$ as well, and using \eqref{eq: bound exponential sum} it can be bounded by
\begin{align*}
		&\int_{1}^{\infty}x^{-\sigma}\dif\,(S(x,t) - S_{c}(x,t)) = \sigma\int_{1}^{\infty}x^{-\sigma-1}(S(x,t) - S_{c}(x,t))\dif x \\
	\ll 	&\int_{1}^{\infty}x^{-\sigma-1/2}\biggl(1+\sqrt{\frac{\log(\,\abs{t}+1)}{\log x}}\biggr)\dif x \ll \frac{1}{\sigma-1/2} + \sqrt{\frac{\log(\,\abs{t}+1)}{\sigma-1/2}},
\end{align*}
where have used the same notation as in the proof of Theorem \ref{th: discretization} for the exponential sums and integrals. Hence for $\sigma>1/2$ and some constant $C>0$
\begin{equation}
\label{eq: bound logzeta sigma>1/2}
	\abs{Z(s)} \le C\biggl(\frac{1}{\sigma-1/2} + \sqrt{\frac{\log(\,\abs{t}+1)}{\sigma-1/2}}\biggr).
\end{equation}
Let now $x$ be large but fixed. We want to derive an estimate for $N(x)$ by Perron inversion. Actually we will apply the Perron formula to $N_{1}(x)\coloneqq \int_{1}^{x}N(u)\dif u$, because then the Perron integral is absolutely convergent. Indeed, we have for any $\kappa>1$ that
\[
	N_{1}(x) = \frac{1}{2\pi\I}\int_{\kappa-\I\infty}^{\kappa+\I\infty}\frac{x^{s+1}\zeta_{\MP}(s)}{s(s+1)}\dif s = \frac{1}{2\pi\I}\int_{\kappa-\I\infty}^{\kappa+\I\infty}\frac{x^{s+1}\e^{Z(s)}}{(s-1)(s+1)}\dif s.
\]
One then uses the fact that $N$ is non-decreasing, so that 
\[	N_{1}(x) - N_{1}(x-1) \le N(x) \le N_{1}(x+1) - N_{1}(x).\] Set $\sigma_{x}=1/2 + (\log x)^{-1/3}$. Then uniformly for 
$\sigma\ge\sigma_{x}$,
 \[
 	\abs{Z(s)} \le C\bigl((\log x)^{1/3}+(\log x)^{1/6}\sqrt{\log(\,\abs{t}+1)}\bigr).
\]
We shift the contour to the line $\sigma=\sigma_{x}$. By the residue theorem (recall that $Z(1)=0$):
\begin{align*}
	N(x)	&\le \frac{(x+1)^{2}}{2} - \frac{x^{2}}{2} + \frac{1}{2\pi\I}\int_{\sigma_{x}-\I\infty}^{\sigma_{x}+\I\infty}\frac{((x+1)^{s+1}-x^{s+1})\e^{Z(s)}}{(s-1)(s+1)}\dif s \\
		&= x + \frac{1}{2} + \frac{1}{2\pi\I}\int_{\sigma_{x}-\I\infty}^{\sigma_{x}+\I\infty}\frac{((x+1)^{s+1}-x^{s+1})\e^{Z(s)}}{(s-1)(s+1)}\dif s.
\end{align*}
We split the range of integration into two pieces: $\abs{t}\le x$ and $\abs{t}>x$. In the first piece we bound $(x+1)^{s+1} - x^{s+1}$ by $\abs{s+1}x^{\sigma_{x}}$, whereas in the second one by $x^{\sigma_{x}+1}$. This gives
\begin{multline*}
	N(x) \le x + \frac{1}{2} + O\biggl\{x^{1/2}\exp\bigl((\log x)^{2/3}\bigr)\int_{0}^{x}\frac{\exp\bigl(2C(\log x)^{2/3}\bigr)}{t+1}\dif t \\
			+ x^{3/2}\exp\bigl((\log x)^{2/3}\bigr)\int_{x}^{\infty}\exp\bigl(2C(\log x)^{1/6}\sqrt{\log(t+1)}\bigr)\frac{\dif t}{t^{2}}\biggr\}
\end{multline*}
The first integral is bounded by $\exp\bigl(2C(\log x)^{2/3}\bigr) \log x$ and the second one by the term $x^{-1}\exp\bigl(2C(\log x)^{2/3}\bigr)$. A similar reasoning applies for a lower bound for $N$, and one sees that the asymptotic relation in \eqref{eq: behavior N and M} for $N$ holds with any $c>2C+1$.

To obtain the asymptotic behavior of $M$, we apply the same reasoning to $N(x) + M(x)$, which is also non-decreasing, and which has Mellin transform
\[
	\zeta_{\MP}(s) + \frac{1}{\zeta_{\MP}(s)} = \frac{s}{s-1}\e^{Z(s)} + \frac{s-1}{s}\e^{-Z(s)},
\]
to show that $N(x) + M(x) = x + O\bigl(x^{1/2}\exp(c(\log x)^{2/3}\bigr)$. The bound for $M$ in \eqref{eq: behavior N and M} then follows by combining this asymptotic estimate with that we have already obtained for $N$.

Finally, the oscillation estimates \eqref{eq: Omega-behavior N and M} follow at once from \eqref{eq: behavior N and M} and the result of Hilberdink and Neamah from \cite{NeamahHilberdink} quoted above.

\end{proof}

\begin{remark}
We stress that the strong bound $\Pi_{\MP}(x) - \Li(x) \ll \log\log x$ is crucial in the above arguments to generate the oscillation estimates \eqref{eq: Omega-behavior N and M}. In particular, if only the weaker bound $\Pi_{\MP}(x) - \Li(x) \ll \sqrt{x}$  had been known (like in Zhang's generalized number system from  \cite[Theorem 1]{Zhang2007}, whose construction is based upon application of the DMVZ-method),  the Hilberdink and Neamah theorem could not have been used to exclude the possibility that the abscissa of convergence $\sigma_c$ of $\sum_{k=1}^{\infty} \mu(n_{k}) n_{k}^{-s}$  satisfies $\sigma_{c}< 1/2$ and that $1/\zeta_{\MP}(s)$ has additional zeros $s=\sigma+\I t$ with $\sigma_{c}<\sigma\le1/2$.
\end{remark}

\begin{remark} Let $\MP$ be a generalized prime number system like in Theorem \ref{th: [0,1/2]-system}. Another example of a general Dirichlet series with abscissa of convergence $1/2$ and with a unique zero in the half plane $\sigma>1/2$ is that of the Liouville function associated with the generalized number system. Its Liouville function, with sum function $L_{\MP}(x)=\sum_{n_{k}\leq x} \lambda(n_k)$, can be defined via the identity
$$
\sum_{k=0}^{\infty} \frac{\lambda(n_k)}{n_{k}^{s}}=\frac{\zeta_{\MP}(2s)}{\zeta_{\MP}(s)},
$$
so that its Dirichlet series has a zero at $s=1$.
Clearly, we have
$$
L_{\MP}(x)=\sum_{n^{2}_{k}\leq x} M_{\MP}(x/n^{2}_{k})\ll x^{1/2}\exp(c(\log x)^{2/3}) \sum_{n_{k}\leq \sqrt{x}}\frac{1}{n_{k}}\ll x^{1/2}\exp(c(\log x)^{2/3}) \log x.
$$
Furthermore, the estimate \eqref{eq: behavior Pi} and (the proof of) \cite[Proposition 19]{Neamah2020} imply
$$
L_{\MP}(x)=\Omega(\sqrt{x}),
$$
 which completes the proof of our claim.
\end{remark}

\begin{remark}
The bound $\Pi_{\MP}(x) - \Li(x) \ll \log\log x$ implies that $\zeta_{\mathcal{P}}$ has meromorphic continuation to $\sigma >0$, and that it has one simple pole at $s=1$ and no other zeros there. The equality $\beta=\gamma=1/2$ implies that both $\zeta_{\mathcal{P}}$ and $1/\zeta_{\mathcal{P}}$ must have infinite order in the strip $0<\sigma<1/2$. (However, using convexity arguments one might show that $\zeta_{\MP}$ and $1/\zeta_{\MP}$ are of polynomial growth in the region $\sigma > 1/2 - 1/\log(\,\abs{t}+2)$.)
\end{remark}

\section{A Beurling number system with highly regular primes but integers with large oscillation}
\label{sec: third application}

In \cite{BrouckeDebruyneVindas2020}, the authors showed the existence of a Beurling prime system $\mathcal{P}$ for which
\[
	\pi_{\MP}(x) = \Li(x) + O(\sqrt{x}) \quad \text{and} \quad N_{\MP}(x) = \rho x + \Omega_{\pm}\bigl(x\exp(-c\sqrt{\log x\log\log x})\bigr)
\]
for any $c>2\sqrt{2}$,
where $\rho>0$ is the asymptotic density of $N_{\MP}$. This was done by first considering a continuous number system $(\Pi_{c}, N_{c})$, for which $\Pi_{c}$ and $N_{c}$ have the desired asymptotic behavior, and then discretizing this continuous system with the aid of a variant of the DMVZ-method, supplemented with a specific technique to control the argument of the zeta function. The continuous prime system is given\footnote{An integer distribution function $N$ is uniquely determined by $\Pi$. Explicitly one has $\dif N = \exp^{\ast}(\dif\Pi)$, where the exponential is with respect to the multiplicative convolution of measures (see e.g.\ \cite[Chapter 3]{DiamondZhangbook}).} by 
\[
	\Pi_{c}(x) = \Li(x) + \sum_{k=0}^{\infty}R_{k}(x), \quad 
	R_{k}(x) = \begin{cases} 
				\sin(\tau_{k}\log x)		& \text{for } \tau_{k}^{1+\delta_{k}} < x \le \tau_{k}^{\nu_{k}},\\
				0					& \text{otherwise}.
			\end{cases}
\]
Here $\{\tau_{k}\}_{k=0}^{\infty}$ is a rapidly increasing sequence, $\delta_{k}=(\log\log \tau_{k} + a_{k})/\log \tau_{k}$, and $\{a_{k}\}_{k=0}^{\infty}$ and $\{\nu_{k}\}_{k=0}^{\infty}\subset (2,3)$ are bounded sequences chosen such that $\Pi_{c}$ is (absolutely) continuous. For a detailed analysis of this continuous example, we refer to \cite{BrouckeDebruyneVindas2020}, where additional technical assumptions are imposed on the sequences $\{\tau_{k}\}_{k=0}^{\infty},$ $\{a_{k}\}_{k=0}^{\infty}$, and $\{\nu_{k}\}_{k=0}^{\infty}$ in order to achieve the needed extremal behavior of its zeta function $\zeta_{c}(s)=\int_{1^{-}}^{\infty}x^{-s}\dif N_{c}(x)= \exp\left(\int_{1^{-}}^{\infty}x^{-s}\dif \Pi_{c}(x)\right)$.

We now show,
\begin{theorem}
There exist discrete Beurling prime systems $\mathcal{P}_{1}$ and $\mathcal{P}_{2}$ which satisfy
\[
	\pi_{\MP_1}(x) = \Li(x) + O(1) \quad \mbox{and} \quad \Pi_{\MP_2}(x) = \Li(x) + O(\log\log x),
\]
and for any $c>2\sqrt{2}$, and $j=1,2$,
\begin{equation}
\label{eq: oscillation estimate}
	N_{\MP_j}(x) = \rho_{j}x + \Omega_{\pm}\bigl(x\exp(-c\sqrt{\log x\log\log x})\bigr),
\end{equation}
where $\rho_{j}>0$ is the asymptotic density of the generalized integer counting function $N_{\MP_j}$.
\end{theorem}

\begin{remark}
In view of its closeness to the `most natural' continuous number system $\Pi_{0} (x)= \Li(x)$ and $N_{0}(x)= x$, the primes of the system $\mathcal{P}_{2}$ might be considered to be more regular than those of $\mathcal{P}_{1}$. Note that the zeta function associated to $\mathcal{P}_{2}$ has no zeros in the right half plane $\{s:\:\Re s>0\}$, as is also the case for the number system from Theorem \ref{th: [0,1/2]-system}.
\end{remark}

\begin{proof}
The prime systems $\mathcal{P}_{1}$ and $\mathcal{P}_{2}$ will be obtained by applying (a variant of) Theorem \ref{th: discretization} to $F=\Pi_{c}$ and $F=\pi_{c}$, respectively, where $\pi_{c}$ is such that $\Pi_{c}(x) = \sum_{n=1}^{\infty}\pi_{c}(x^{1/n})/n$. We have the formula
\[
	\pi_{c}(x) = \li(x) + \sum_{k=0}^{\infty}\sum_{n=1}^{\infty}r_{k,n}(x), 
\]
where (cf. Section \ref{sec: First application})
\begin{align*}
	\li(x) 		&= \sum_{n=1}^{\infty}\frac{(\log x)^{n}}{n!n\zeta(n+1)} \quad \mbox{and}\\ 
	r_{k,n}(x) 	&= \frac{\mu(n)}{n}R_{k}(x^{1/n}) 
	= \begin{cases}
		\frac{\mu(n)}{n}\sin\bigl((\tau_{k}/n)\log x\bigr)	&\text{for } (\tau_{k}^{1+\delta_{k}})^{n} < x \le (\tau_{k}^{\nu_{k}})^{n},\\
		0									&\text{otherwise}.
	\end{cases}
\end{align*}

Let us first verify that $\pi_{c}$ satisfies the requirements of Theorem \ref{th: discretization}. We only have to show that it is non-decreasing, the rest is clear. Set $I_{k,n} = \left[ (\tau_{k}^{1+\delta_{k}})^{n}, (\tau_{k}^{\nu_{k}})^{n}\right]$, and let $x\ge 1$ be fixed. We have the following elementary observations:
\begin{enumerate}
	\item[a)] for every $n\ge 1$ there exists at most one $k$ such that $x\in I_{k,n}$;
	\item[b)] if $n_{1}\le n_{2}$ and $k_{1}, k_{2}$ are such that $x\in I_{k_{1}, n_{1}} \cap I_{k_{2}, n_{2}}$, then $k_{1}\ge k_{2}$.
\end{enumerate}
The first observation is a consequence of $I_{k, 1} \cap I_{k+1, 1} = \varnothing$, which holds provided that the sequence $\{\tau_{k}\}_{k=0}^{\infty}$ grows sufficiently rapidly. The second property follows from the fact that if $k_{1}<k_{2}$, then $\tau_{k_{1}}^{\nu_{k_{1}}n_{1}} < \tau_{k_{2}}^{(1+\delta_{k_{2}})n_{1}} \le \tau_{k_{2}}^{(1+\delta_{k_{2}})n_{2}}$. Now $\pi_{c}$ is absolutely continuous, so it will follow that it is non-decreasing if we show that $\pi_{c}'$ is non-negative. We have
\[
	\abs[3]{\biggl(\sum_{k=0}^{\infty}\sum_{n=1}^{\infty}r_{k,n}(x)\biggr)'} \le \frac{1}{x}\sum_{\substack{k,n \\ x\in I_{k,n}}}\frac{\tau_{k}}{n^{2}}.
\]
Let $m$ be the smallest integer $\ge1$ such that $x\in I_{k,m}$ for some $k$, and let $K$ be such that $x\in I_{K,m}$. By a) and b), the above quantity is bounded by 
\[
	\frac{1}{x}\sum_{n=m}^{\infty}\frac{\tau_{K}}{n^{2}} \le \frac{2\tau_{K}}{mx} \le \frac{2}{m}\tau_{K}^{-\delta_{K}} = \frac{2}{m\e^{a_{K}}\log\tau_{K}},
\]
where the last inequality follows from the fact that $x\in I_{K,m}$. Also, 
\[
	\li'(x) \ge \frac{1}{\zeta(2)}\frac{1-x^{-1}}{\log x} \ge \frac{1}{2\zeta(2)\log x} \ge \frac{1}{2m\zeta(2)\nu_{K}\log\tau_{K}},
\]
if $x\in I_{K,m}$. Hence $\pi_{c}'(x) \ge 0$ if we impose that $a_{k}\ge \log(12\zeta(2))$ say (recall that $\nu_{k}\le 3$).	

If we now apply Theorem \ref{th: discretization} with $F = \Pi_{c}$ and $F = \pi_{c}$, we obtain two prime systems $\MP_{1}$ and $\MP_{2}$ with $\pi_{\MP_1}(x) = \Pi_{c}(x) + O(1)=\Li(x)+O(1)$ and $\pi_{\MP_2}(x) = \pi_{c}(x) + O(1)$, so that as in the proof of Theorem \ref{th: [0,1/2]-system}, $\Pi_{\MP_2}(x) = \Pi_{c}(x) + O(\log\log x) = \Li(x) + O(\log\log x)$.
To deduce the oscillation estimate \eqref{eq: oscillation estimate} from the behavior of the zeta functions $\zeta_{\MP_1}$ and $\zeta_{\MP_2}$, two additional properties alongside \eqref{eq: bound exponential sum} were required, namely that
\[
	S(x,t) -S_{c}(x,t) \ll \sqrt{x}(\log\tau_{k})^{1/4}, \quad \text{uniformly for } \abs{t-\tau_{k}} \le \exp\biggl(d\biggl(\frac{\log\tau_{k}}{\log\log\tau_{k}}\biggr)^{1/3}\biggr),
\]	
for a certain specified constant $d>0$ and $x$ and $k$ large enough, and that 
\[
	\abs{\arg \frac{\zeta_{\mathcal{P}}(1+\I\tau_{k_{l}})}{\zeta_{c}(1+\I\tau_{k_{l}})}} <  \frac{\pi}{80}, \quad \mbox{on an (infinite) subsequence $\{\tau_{k_{l}}\}$ of $\{\tau_{k}\}$}.
\]
Both of these additional properties can be obtained as in \cite{BrouckeDebruyneVindas2020}. For the first one, one modifies the proof of Theorem \ref{th: discretization} and also considers the events $B_{k,j}$, corresponding to the violation of the above bound for $x=q_{j}$ and $t$ close to $\tau_{k}$. Using the rapid growth of the sequence $\{\tau_{k}\}_{k=0}^{\infty}$, one then shows that the sum of the probabilities of the events $B_{k,j}$ is also finite. For the second property, one adds a finite number of well chosen primes around $80/\pi$, which shift the phase of the zeta function at the points $1+\I\tau_{k_{l}}$ to the desired range. We choose to omit further details and instead refer to \cite[Section 6]{BrouckeDebruyneVindas2020} for an account on both methods.
\end{proof}


\begin{thebibliography}{99}

\bibitem{BatemanDiamond} P.~T.~Bateman, H.~G.~Diamond, \emph{Asymptotic distribution of Beurling's generalized prime numbers}, in:
\emph{Studies in number theory}, W. J. LeVeque (ed.), pp. 152--210. Math. Assoc. Amer., Prentice-Hall, Englewood Cliffs, N.J., 1969.

\bibitem{BGTbook}
N.~Bingham, C.~Goldie, and J.~Teugels, \emph{Regular variation},
  Encyclopedia of Mathematics and its Applications, 27, Cambridge University Press, Cambridge, 1987.
  

\bibitem{Beurling1937} A.~Beurling, \textit{Analyse de la loi asymptotique de la distribution des nombres premiers g\'{e}n\'{e}ralis\'{e}s,} Acta Math. \textbf{68} (1937), 255--291.

\bibitem{BrouckeDebruyneVindas2020} F.~Broucke, G.~Debruyne, J.~Vindas, \emph{Beurling integers with RH and large oscillation}, Adv.\ Math. \textbf{370} (2020), Article 107240.

    
\bibitem{DiamondMontgomeryVorhauer} H.~G.~Diamond, H.~L.~Montgomery, U.~M.~A.~Vorhauer, \emph{Beurling primes with large oscillation,} Math. Ann. \textbf{334} (2006), 1--36.

\bibitem{DiamondZhangbook} H.~G.~Diamond, W.-B.~Zhang, \emph{Beurling
    generalized numbers,} Mathematical Surveys and Monographs series,
  American Mathematical Society, Providence, RI, 2016. 

\bibitem{Hilberdink2005} T.~W.~Hilberdink, \emph{Well-behaved Beurling primes and integers}, J.\ Number Theory  \textbf{112} (2005), 332--344.

\bibitem{HilberdinkLapidus2006} T.~W.~Hilberdink, M.~L.~Lapidus, \emph{Beurling zeta functions, generalised primes, and fractal membranes}, Acta Appl.\ Math \textbf{94} (2006), 21--48.

\bibitem{Knopfmacher} J.~Knopfmacher, \emph{Abstract analytic number theory,} North-Holland Publishing Co., Amsterdam-Oxford, 1975.

\bibitem{Landau1903} E.~Landau, \emph{Neuer Beweis des Primzahlsatzes und Beweis des Primidealsatzes,} Math. Ann. \textbf{56} (1903), 645--670.

\bibitem{Loeve} M.~Lo\`eve, \emph{Probability theory. I}, Graduate Texts in Mathematics 45, Fourth edition, Springer-Verlag, New York-Heidelberg, 1977.


\bibitem{MontgomeryVorhauer2000} H.~L.~Montgomery, U.~M.~A.~Vorhauer, \emph{Some unsolved problems in harmonic analysis as found in analytic number theory}, 2000, \url{http://www.prometheus-us.com/asi/analysis2000/papers/montgomery_problems.pdf}

\bibitem{Neamah2020} A.~A.~Neamah, \emph{Completely multiplicative functions with sum zero over generalised prime systems,} Res. number theory \textbf{6} (2020), Article 45.


\bibitem{NeamahHilberdink} A.~A.~Neamah, T.~W.~Hilberdink, \emph{The average order of the M\"obius function for Beurling primes}, Int.\ J.\ Number Theory \textbf{16} 05 (2020), 1005--1011.

\bibitem{Zhang2007}  W.-B.~Zhang, \emph{Beurling primes with RH and Beurling primes with large oscillation,} Math. Ann. \textbf{337} (2007), 671--704.

\end{thebibliography}
\end{document}